\documentclass[10pt,a4paper]{article}
\usepackage{amsmath, amsthm, amsfonts, amssymb}
\usepackage{authblk}
\usepackage[utf8]{inputenc}
\usepackage{enumitem}
\usepackage[usenames,dvipsnames]{xcolor}
\usepackage{tikz-cd}
\usepackage{parskip}

\def\C{\mathbb{C}}
\def\N{\mathbb{N}}

\def\M{\mathcal{M}}

\def\res{\operatorname{res}}
\def\Ann{\operatorname{Ann}}
\def\Hess{\operatorname{Hess}}

\theoremstyle{plain}
\newtheorem{thm}{Theorem}[section]
\newtheorem{lemma}[thm]{Lemma}
\newtheorem{prop}[thm]{Proposition}

\theoremstyle{definition}
\newtheorem{defnx}[thm]{Definition}
\newenvironment{defn}
  {\pushQED{\qed}\defnx}
  {\popQED\enddefnx}
\newtheorem{exx}[thm]{Example}
\newenvironment{ex}
  {\pushQED{\qed}\exx}
  {\popQED\endexx}

\theoremstyle{remark}

\newtheorem{que}[thm]{Question}

\title{On binomial complete intersections}
\author[1]{Filip Jonsson Kling\thanks{filip.jonsson.kling@math.su.se}}
\author[1]{Samuel Lundqvist\thanks{samuel@math.su.se}} 
\author[2]{Lisa Nicklasson\thanks{lisa.nicklasson@mdu.se}}
\affil[1]{\small Stockholms universitet}
\affil[2]{\small Università di Genova, Mälardalens universitet}
\date{} 

\begin{document}
\maketitle

\begin{abstract}
\noindent 
We consider homogeneous binomial ideals $I=(f_1,\ldots,f_n)$ in $K[x_1,\ldots ,x_n]$, where  
$f_i = a_i x_i^{d_i} - b_i m_i$ and $a_i \neq 0$. When such an ideal is a complete intersection, we show that the monomials which are not divisible by $x_i^{d_i}$ for $i=1,\ldots,n$ form a vector space basis for the corresponding quotient, and we describe the Macaulay dual generator in terms of a directed graph that we associate to $I$. These two properties can be seen as a natural generalization of well-known properties for monomial complete intersections. Moreover, we give a description of the radical of the resultant of $I$ in terms of the directed graph.
\vspace{20pt}

\noindent {\bf Keywords:} Complete intersection, binomial ideal, Macaulay's inverse system, resultant, term-rewriting\\
{\bf MSC 2020:} 13C40, 13E10, 13F65, 13P15, 16S15
\end{abstract}

\section*{Introduction}

An artinian monomial complete intersection $( x_1^{d_1},\ldots,x_n^{d_n} )$ in $K[x_1, \ldots, x_n]$ enjoys the following two well known properties: the set of monomials which are not divisible by $x_i^{d_i}$ for $i=1,\ldots, n$ is a vector space basis for the corresponding quotient, and the Macaulay dual generator equals $X_1^{d_1-1} \cdots X_n^{d_n-1}$ up to a constant factor.

In this paper we argue that these characteristics should be seen as special cases of properties of complete intersections on the form
$$I = ( a_1 x_1^{d_1} - b_1 m_1, \ldots, a_n x_n^{d_n} - b_n m_n ),$$ where $a_i \neq 0$, and 
$m_i$ is a monomial of degree $d_i$. Our main results are the following.

\begin{enumerate}
    \item Associated to $I$ there is a graph $G$ that gives a rewriting rule transforming any element 
    in the polynomial ring to a linear combination of monomials not divisible by $x_i^{d_i}$ for $i=1,\ldots, n$ modulo $I$. 
    In particular, the monomials not divisible by the $x_i^{d_i}$'s constitute a vector space basis for the quotient ring (Theorem \ref{thm:basis}).
      \item The coefficients of the Macaulay dual generator are monomials in the coefficients of the generators of $I$, and can be described algorithmically in terms of the graph $G$ (Theorem \ref{thm:dual}).
    \item The results above require that $I$ is a complete intersection. The radical of the resultant of $I$, which is a polynomial in the $a_i$'s and $b_i$'s, determines for which choices of coefficients $I$ really is a complete intersection. We determine this polynomial combinatorially in terms of $G$ (Theorem \ref{thm:resultant}) with the restriction that each $m_i$ is not a pure power of a variable. We also give a less precise description of the radical of  the resultant for our general class of ideals (Theorem \ref{thm:new_resultant}).

\end{enumerate}

The first and the third result generalize a construction for the case $d_1 = \cdots = d_n = 2$ by Harima, Wachi, and Watanabe \cite{HWW}, and this paper has served as the main source of inspiration for our study.


\section{A monomial basis for the quotient ring}

Let $R=K[x_1, \ldots, x_n]$ be a polynomial ring with the standard grading where $K$ is any field. Recall that homogeneous polynomials $f_1, \ldots, f_n$ form a complete intersection if they are a regular sequence, or equivalently if their only common zero over the algebraic closure $\overline{K}$ is the origin. A useful tool for illustrating conditions on when a family $f_1, \ldots, f_n$ forms a complete intersection is the \emph{resultant}. The resultant is a polynomial $\res(f_1,\dots, f_n)$ in the coefficients of $f_1,\dots, f_n$  with the property that $\res(f_1,\dots, f_n)\ne 0$ if and only if $f_1,\dots, f_n$ is a complete intersection. 

\begin{defn}\label{def:normal_form}
A set of homogeneous polynomials $\{f_1, \ldots, f_n\}$ of degrees $d_1, \ldots, d_n$ in $R$ is on \emph{normal form} if the coefficient of $x_i^{d_i}$ in $f_i$ is nonzero.
\end{defn}  

Our goal in this section is to describe the vector space structure of the quotient ring $R/(f_1, \ldots, f_n)$ when $(f_1, \ldots, f_n)$ is a complete intersection of binomials on normal form. Along the way we will also discover some properties of the resultant. Here a \emph{binomial} means a polynomial with at most two terms. Notice that we regard a monomial as a binomial.

We remark that our definition of normal form is more permissive than the one given for quadratic polynomials in \cite{HWW}, as they require $f_i$ to be a linear combination of $x_i^2$ and squarefree monomials. In the case of binomials, Definition \ref{def:normal_form} agrees with the homogeneous instance of the definition of binomials on normal form in \cite{CA}, although \cite{CA} focuses on the non-homogeneous case. 

We are interested in families of binomials on normal form, in the following sense.
 
\begin{defn}\label{def:family}
A \emph{family of binomials on normal form} is a $2n$-parameter family 
\[ \{a_1x_1^{d_1}-b_1m_1, \  \ldots,\  a_nx_n^{d_n}-b_nm_n\} \] 
of homogeneous binomials on normal form, where the monomials $m_1, \ldots, m_n$ are fixed,  $a_1, \ldots, a_n, b_1, \ldots, b_n$ are parameters, and $a_i \ne 0$ for $i=1, \ldots, n$.
\end{defn}

We remark that all our results are valid when substituting $a_i=1$, but we found that the above notation makes some of our arguments more clear.

We may assume that $m_i \ne x_i^{d_i}$ for $i=1, \ldots, n$ in Definition \ref{def:family}.
As a tool for studying families of binomials on normal form we introduce a directed graph. 

\begin{defn}\label{def:graph}
Given a family $B=\{f_1,\dots, f_n\} \subset R$ of binomials  $f_i=a_ix_i^{d_i} - b_im_i$ on normal form, and an integer $d>0$, we let the \emph{reduction graph} $G_{B,d}$ be the directed graph with vertex set
\[
\{ x_1^{\alpha_1} \cdots x_n^{\alpha_n} \ : \ \alpha_1 + \dots + \alpha_n=d\} \subset R
\]
and labeled edge set
\[
\Big\{ m \overset{i}{\to} \frac{m_im}{x_i^{d_i}} \ : \ x_i^{d_i}|m, \ \text{and} \ x_j^{d_j}\nmid m \ \text{for} \ j<i\Big\}. \qedhere
\]
\end{defn}

Observe that if $m \overset{i}{\to}m'$ is an edge in $G_{B,d}$ then 
\begin{equation*}
    a_im-\frac{m}{x_i^{d_i}}f_i=b_im'.
\end{equation*}
Hence if there is a directed path from a monomial $m$ to a monomial $m'$ in the reduction graph, then $m$ will be a scalar multiple of $m'$ modulo $(B)$.

For positive integers $d_1, \ldots, d_k$, where $k \le n$, let $\M_{d_1, \ldots, d_k}$ denote the set of monomials $x_1^{\alpha_1} \cdots x_n^{\alpha_n}$ with $\alpha_i<d_i$ for each $i\le k$. Recall that a \emph{sink} in a directed graph is a vertex with no outgoing edges.
Notice that a monomial $m$ of degree $d$ is a sink in $G_{B,d}$ if it belongs to $\M_{d_1, \ldots, d_n}$, and has precisely one outgoing edge  otherwise. In this way, the reduction graph $G_{B,d}$ gives a rewriting rule for how a monomial $m$ of degree $d$ can be represented by a constant multiple of a monomial in $\M_{d_1, \ldots, d_n}$ modulo $(B)$ if the path from $m$ ends at a sink. If the path starting at $m$ does not end, then it must reach a cycle as $G_{B,d}$ only has finitely many vertices. In that case we will show in Lemma \ref{lem:monomial} that  $m\in (B)$, when $(B)$ is a complete intersection. 

\begin{defn}
To a directed cycle $C$ in the reduction graph $G_{B,d}$, we associate the \emph{cycle polynomial} $p(C)$ in $a_1, \ldots, a_n$, $b_1, \ldots, b_n$ as
\[
p(C)=a_1^{r_1}\cdots a_n^{r_n} - b_1^{r_1} \cdots b_n^{r_n}
\]
where $r_i$ is the number of $i$-labeled edges in $C$. We then define the cycle polynomial of the whole reduction graph as
\begin{equation*}
p(G_{B,d})=\prod_{C} p(C)
\end{equation*}
where the product is taken over all directed cycles $C$ in $G_{B,d}$. 
\end{defn}

For a polynomial $p$ we let $\sqrt{p}$ denote its radical, i.\,e.\ the product of the distinct irreducible factors of $p$. 
In Lemma \ref{lem:monomial} as well as in Section \ref{sec:res} we discuss irreducible factors and divisibility of the cycle polynomial and the resultant of a family of binomials on normal form. In this context the cycle polynomial and the resultant should be considered elements of a polynomial ring $K[a_1, \ldots, a_n, b_1, \ldots, b_n]$.

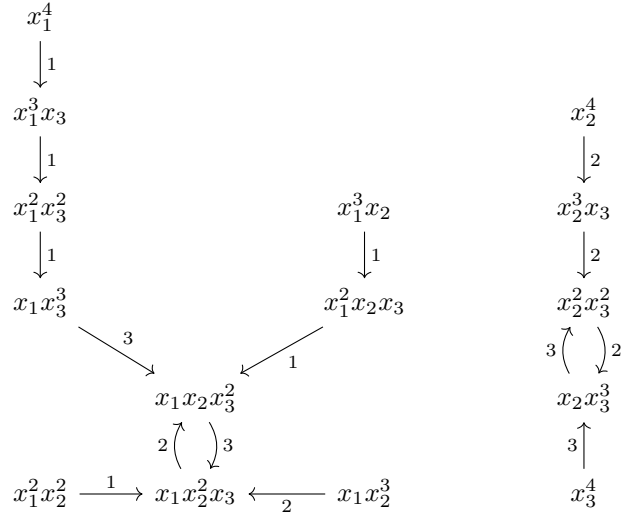
\begin{figure}[ht]
    \centering
    \[\begin{tikzcd}
x_1^4 \arrow[d, "1"]      &                                       &                             &  &                                      \\
x_1^3x_3 \arrow[d, "1"]   &                                       &                             &  & x_2^4 \arrow[d, "2"]                 \\
x_1^2x_3^2 \arrow[d, "1"] &                                       & x_1^3x_2 \arrow[d, "1"]     &  & x_2^3x_3 \arrow[d, "2"]              \\
x_1x_3^3 \arrow[rd, "3"]  &                                       & x_1^2x_2x_3 \arrow[ld, "1"] &  & x_2^2x_3^2 \arrow[d, "2", bend left] \\
                          & x_1x_2x_3^2 \arrow[d, "3", bend left] &                             &  & x_2x_3^3 \arrow[u, "3", bend left]   \\
x_1^2x_2^2 \arrow[r, "1"] & x_1x_2^2x_3 \arrow[u, "2", bend left] & x_1x_2^3 \arrow[l, "2"]     &  & x_3^4 \arrow[u, "3"]                
\end{tikzcd}\]
    \caption{$G_{B,4}$ for $B$ as in Example \ref{ex:begining_example}}
    \label{fig:graphexample}
\end{figure}

\begin{ex}\label{ex:begining_example}
Consider $f_1=a_1x_1^2 - b_1x_1x_3$, $f_2=a_2x_2^2 - b_2x_2x_3$, and $f_3=a_3x_3^2 - b_3x_2x_3$. Then $B=\{f_1,f_2,f_3\}$ is a family of binomials on normal form and the reduction graph $G_{B,4}$ is given in Figure \ref{fig:graphexample}. Here there are two cycles, one from $x_1x_2x_3^2$ to $x_1x_2^2x_3$ and back, and one from $x_2^2x_3^2$ to $x_2x_3^3$ and back, both having a $2$-labeled edge and a $3$-labeled edge, so the cycle polynomial is $p(G_{B,4})=(a_2a_3-b_2b_3)^2$ and $\sqrt{p(G_{B,4})}=a_2a_3-b_2b_3$. Note that any monomial in the graph has a path leading to a cycle. This is expected as $R/(B)$ has socle degree $3$ when $B$ is a complete intersection, so any monomial of degree $4$ should vanish.
\end{ex} 
Now we are ready to prove the key lemma of this section.
\begin{lemma}
\label{lem:monomial}
Let $B=\{f_1, \ldots, f_n\}$ be a family of binomials $f_i=a_ix_i^{d_i}-b_im_i$ on normal form, and let $d$ be a positive integer. Then 
\[
\sqrt{p(G_{B,d})}\  {\Big |} \ \res(B).
\]
Moreover, fix an integer $1 \le k \le n$. For each monomial $m \notin \M_{d_1, \ldots, d_k}$, the directed path starting at the vertex $m$ in the reduction graph either leads to
\begin{enumerate}
\item a monomial $m' \in \M_{d_1, \ldots, d_k}$, in which case $$m \equiv \frac{b_1^{r_1} \cdots b_n^{r_n}}{a_1^{r_1} \cdots a_n^{r_n}} m' \mod (f_1, \ldots, f_k)$$ 
where $r_i$ is the number of $i$-labeled edges in the directed path,
or
\item a directed cycle, and $m \in (f_1, \ldots, f_k)$ when $f_1, \ldots, f_k$ is a regular sequence. 
\end{enumerate}
\end{lemma}

\begin{proof}
First, recall that if $m \overset{i}{\to}m'$ is an edge in $G_{B,d}$ then 
\begin{equation*}
    a_im-\frac{m}{x_i^{d_i}}f_i=b_im'.
\end{equation*}
Moreover, a directed path
\[
m^{(0)} \overset{i_1}{\longrightarrow} m^{(1)} \overset{i_2}{\longrightarrow} m^{(2)} \overset{i_3}{\longrightarrow}  \cdots \overset{i_r}{\longrightarrow} m^{(r)} 
\]
gives rise to the relation
\begin{align}\label{eq:path_rel}
 & a_{i_1} \cdots a_{i_r} m^{(0)} - \sum_{s=1}^r p_s \frac{m^{(s-1)}}{x_{i_s}^{d_{i_s}}}f_{i_s} = b_{i_1} \cdots b_{i_r} m^{(r)}
 \\ & \text{where} \ p_s= \prod _{\ell>s}a_{i_\ell} \prod_{\ell<s} b_{i_\ell}. \nonumber
\end{align}

We shall prove the theorem by induction over $k$. When $k=1$ it is clear that {\it 1.} holds for any monomial, by reducing modulo $f_1$. So, consider $k$ fixed, and assume {\it 1.} or {\it 2.} holds true for all monomials, for any smaller value of $k$. 
Take a monomial $m \notin \M_{d_1, \ldots, d_k}$, and consider the vertex $m$ in the reduction graph $G_{B,d}$ where $d=\deg m$. There is an outgoing edge with label $\le k$ from $m$, as otherwise $m$ would belong to the set $\M_{d_1, \ldots, d_k}$. We follow the directed path starting in $m$, as long as the edges have labels $\le k$. If this path leads to a monomial $m' \in \M_{d_1, \ldots, d_k}$ then \eqref{eq:path_rel} gives precisely the expression in {\it 1.} If not, we follow the path until we reach a monomial $m'$ that we have already seen. In this case, we have found a directed cycle $C=m' \to \dots \to m'$. By \eqref{eq:path_rel} we get
\[
a_1^{r_1} \cdots a_n^{r_n} m'-(h_1f_1 + \dots + h_kf_k)=b_1^{r_1} \cdots b_n^{r_n}m'.
\]
where $r_i$ is the number of $i$-labeled edges in $C$. Equivalently
\[
p(C)m'=h_1f_1 + \dots + h_kf_k.
\]
Here $h_i$ is linear combinations of unique monomials in $\M_{d_1, \ldots, d_{i-1}}$ coming from the vertices of $C$, and the coefficients are monomials in $a_1, \ldots, a_n, b_1, \ldots, b_n$, as described in \eqref{eq:path_rel}.

We now claim that a zero of $p(C)$ in $\overline{K}^n$ is also a zero of $\res(B)$. If the claim holds true, we can conclude that $m'\in (B)$ when $(B)$ is a regular sequence. To prove the claim, consider a point for which $p(C)=0$. As the polynomial $p(C)$ is not divisible by any $a_i$ or $b_i$, we may choose nonzero values for $a_1, \ldots, a_n, b_1, \ldots, b_n$. With this choice of point we have
\[
h_1f_1 + \dots + h_kf_k=0
\]
and no cancellation of terms in $h_1, \ldots, h_k$. Note that we do have $h_i=0$ if and only if there is no $i$-labeled edge in $C$. To produce a cycle we must have at least two different edge labels, so we have  $h_1f_1 + \dots + h_\ell f_\ell=0$ for some $1<\ell \le k$ with $h_\ell \ne 0$.
As $h_\ell$ is a linear combination of monomials in $\M_{d_1, \ldots, d_{\ell-1}}$ we have, by the inductive assumption, that $h_\ell \ne 0$ in $R/(f_1, \ldots, f_{\ell-1})$ if $f_1, \ldots, f_{\ell-1}$ is a regular sequence. But then $h_\ell$ is a zerodivisor of $f_\ell$ in   $R/(f_1, \ldots, f_{\ell-1})$, so $f_1, \ldots, f_{\ell}$ is not a regular sequence, and therefore $\res(B)=0$. 


Applying the above argument to a monomial in a cycle $C$ shows that $\sqrt{p(C)}$ divides $\res(B)$, and hence $\sqrt{p(G_{B,d})}$ divides $\res(B)$. 
\end{proof}

\begin{ex}\label{ex:path}
Consider $B=\{f_1,f_2,f_3\}$ for $f_1=a_1x_1^2 - b_1x_1x_2$, $f_2=a_2x_2^2 - b_2x_1x_3$, and $f_3=a_3x_3^2 - b_3x_1^2$. Then $B$ is a family of binomials on normal form with graph $G_{B,3}$ as given in Figure \ref{fig:graphexample_path}. Following the directed path from $x_1^2x_2$, we end up at $x_1x_2x_3\in \M_{2,2,2}$ and get that
\[
x_1^2x_2 \equiv \frac{b_1}{a_1}x_1x_2^2 \equiv \frac{b_1b_2}{a_1a_2}x_1^2x_3 \equiv \frac{b_1^2b_2}{a_1^2a_2}x_1x_2x_3 \mod (B).
\]
As every vertex in $G_{B,3}$ has a directed path leading to $x_1x_2x_3$, we can do a similar rewriting for any degree three monomial in $K[x_1,x_2,x_3]$ to show that such a monomial is equal to a multiple of $x_1x_2x_3$ modulo $(B)$, which moreover will be nonzero if all $b_i\neq 0$. This can be contrasted to Figure \ref{fig:graphexample} where the reduction graph leads every monomial of degree four to a directed cycle and to Figure \ref{fig:graphexample_deg3} where some monomials have a directed path to a monomial in $\M_{2,2,2}$ while others have a path leading to a directed cycle.
\end{ex}

\begin{figure}[ht]
    \centering
    \begin{tikzcd}
                         & x_2^2x_3 \arrow[d, "2"] &                       \\
                         & x_1x_3^2 \arrow[d, "3"] &                       \\
x_2x_3^2 \arrow[rd, "3"] & x_1^3 \arrow[d, "1"] &  \\
                         & x_1^2x_2 \arrow[d, "1"] &                       \\
                         & x_1x_2^2 \arrow[d, "2"] &   x_3^3 \arrow[ld, "3"]                    \\
x_2^3 \arrow[rd, "2"]   & x_1^2x_3 \arrow[d, "1"] &  \\
                        & x_1x_2x_3               &                      
\end{tikzcd}
    \caption{$G_{B,3}$ for $B$ as in Example \ref{ex:path}}
    \label{fig:graphexample_path}
\end{figure}
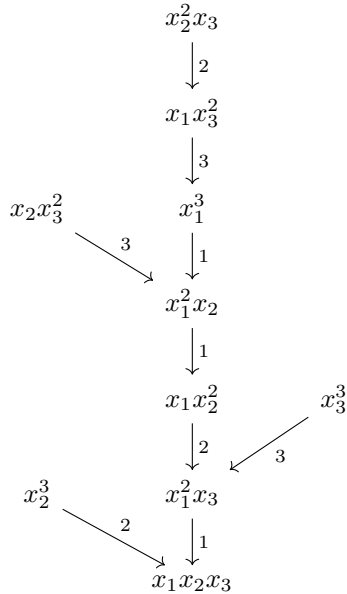

As mentioned in the introduction, the set of monomials $\M_{d_1, \ldots, d_n}$ is a vector space basis for the monomial complete intersection $R/(x_1^{d_1}, \ldots, x_n^{d_n})$, and it is proved in \cite{HWW} that the set $\M_{2, \ldots, 2}$ of squarefree monomials is a basis for $R/(f_1,\ldots, f_n)$ when $(f_1,\ldots, f_n)$ is a quadratic binomial complete intersection, where $f_i = a_i x_i^2 + b_i m_i$ with $a_i \neq 0$ and $m_i$ a squarefree monomial.  
By Lemma \ref{lem:monomial} this now generalizes to any choice of positive integers $d_1, \ldots, d_n$ as well as to our more general normal form.

\begin{thm}\label{thm:basis}
Let $B$ be a complete intersection of binomials of degrees $d_1$, ..., $d_n$ on normal form. Then $\M_{d_1, \ldots, d_n}$ is a vector space basis for $R/(B)$. 
\end{thm}
\begin{proof}
It follows from Lemma \ref{lem:monomial} with $k=n$ that any polynomial in $R/(B)$ can be expressed as a linear combination of monomials in $\M_{d_1, \ldots, d_n}$, and it is well known that the vector space dimension of $R/(B)$ is 
 $|\M_{d_1, \ldots, d_n}|$.   
\end{proof}

Recall that the most common way to determine a monomial basis for a quotient ring $R/I$ is to compute the initial ideal of $I$ w.\,r.\,t.\ some term order. The monomials outside the initial ideal then constitute a vector space basis for $R/I$. 
We remark, however, that the set of monomials in $\M_{d_1, \ldots, d_n}$ cannot in general be obtained from a Gr\"obner basis computation, or more precisely, that the monomials $x_1^{d_1},\ldots,x_n^{d_n}$ cannot in general belong to the initial ideal of $I$ with respect to a term order. Indeed, if $b_i \ne 0$ for $ i = 1, \ldots, n$, and we assume a term order for which $x_n$ is the least variable, then $x_n^{d_n}$ is the least monomial of degree $d_n$, and thus $m_n > x_n^{d_n}.$

\section{The Macaulay dual generator}

Recall that artinian complete intersections are Gorenstein, and that an artinian Gorenstein algebra can be defined in terms of a \emph{Macaulay dual generator}, unique up to a multiplicative constant. The general setup is by letting $R =K[x_1, \ldots, x_n]$ act by contraction on the polynomial ring $S=K[X_1, \ldots, X_n]$ by defining $x_i^a \circ X_i^b = X_i^{b-a}$ when $a \leq b$ and $0$ otherwise. Over a field of 
of characteristic zero, the construction is usually done by letting $R$ act by differentiation on $S$, $x_i= \frac{\partial}{\partial X_i}$. 
Throughout this section we will use the differentiation operator, but we will state the main result in terms of the contraction operator as well. Thus, we temporarily assume $K$ to be of characteristic zero.

For a polynomial $F \in S$ we define its annihilator ideal as
\[
\Ann(F)=\{ f \in R \ : \ f \circ F=0\}.
\]
Macaulay's Double annihilator Theorem states that a standard graded artinian algebra $A=\bigoplus_{i=0}^d A_i$ is Gorenstein if and only if there is a homogeneous polynomial $F \in S_d$ such that $A \cong R/\Ann(F)$. 

Given a Gorenstein algebra $R/(f_1, \ldots, f_r)$, one can determine the dual generator $F$ by solving the system of PDE's that $f_i \circ F=0$ impose. For example, this can be done in Macaulay2 \cite{M2} using the command \verb|toDual|.

As mentioned in the introduction, the monomial $X_1^{d_1-1} \cdots X_n^{d_n-1}$ is the dual generator for a monomial complete intersection $(x_1^{d_1}, \ldots, x_n^{d_n})$. In general, results providing an explicit expression for the Macaulay dual generator of a complete intersection, given a generating set on a certain form, are rare in the literature. We now aim to give such a result for our object of study.

So, given a binomial complete intersection on normal form given by $B=\{f_1, \ldots, f_n\}$ where $f_1=a_1x_1^{d_1} - b_1m_1,\dots, f_n=a_nx_n^{d_n} - b_nm_n$, with socle degree $d=d_1+\cdots + d_n - n$, we know that it has a dual generator of the form 
\[
F=\sum_{\alpha}c_{\alpha}X^{\alpha}
\]
where the sum runs over all $\alpha\in \N^n$ with $\alpha_1 + \cdots + \alpha_n=d$. What criteria on the coefficients $c_{\alpha}$ of $F$ does $f_i \circ F=0$ impose? For every monomial $X^{\gamma}$ of degree $d-d_i$ and monomial $m$ of degree $d_i$, there exists a unique monomial $X^{\alpha}$ of degree $d$ such that $m \circ X^{\alpha}=\ell X^{\gamma}$ for some nonzero $\ell \in \N$. Therefore
\[
f_i \circ F = \sum_{\gamma}(ka_ic_{\alpha} - \ell b_ic_{\beta})X^{\gamma}
\]
where the sum is taken over all $\gamma\in \N^n$ with $\gamma_1 + \cdots + \gamma_n=d-d_i$. Here $c_{\alpha}$ is the coefficient of the unique $X^{\alpha}$ satisfying that $x_i^{d_i} \circ X^{\alpha}$ is a nonzero multiple $k$ of $X^{\gamma}$, and similarly $c_{\beta}$ is the coefficient of the unique $X^{\beta}$ satisfying that $m_i \circ X^{\beta}$ is a nonzero multiple $\ell$ of $X^{\gamma}$. Let $\widehat{G}_{B,d}$ be the graph containing $G_{B,d}$ as a subgraph but where we add extra edges
\[
m \overset{i}{\to} \frac{m_im}{x_i^{d_i}}\quad \text{if} \quad x_i^{d_i}|m 
\]
and no additional restrictions on $i$. For all $f_i\circ F$ to vanish, we thus get a relation $k a_i c_{\alpha}= \ell b_i c_{\beta}$
between any two coefficients of monomials of neighbouring vertices in $\widehat{G}_{B,d}$. However, it suffices to look at the relations coming from edges only in the reduction graph $G_{B,d}$. It also gives us enough information to algorithmically determine $F$ completely.

\begin{thm}\label{thm:dual}
Let $B=\{a_1x_1^{d_1} - b_1m_1, \dots, a_nx_n^{d_n} - b_nm_n\}$ be a regular sequence of binomials, and let $d=d_1 + \cdots +d_n - n$. Then the dual generator 
$F=\sum_{\alpha}c_{\alpha}X^{\alpha}$
can be determined from the reduction graph $G_{B,d}$ as follows. Let $s_i$ be the maximal number of $i$-labeled edges in any directed path ending at $x_1^{d_1-1}\cdots x_n^{d_n-1}$.
\begin{itemize}
\item  If there is a (possibly trivial) directed path from $x^{\alpha}$ to $x_1^{d_1-1}\cdots x_n^{d_n-1}$ with $r_i$ $i$-labeled edges, then
\[
c_{\alpha} =  \left\{ \begin{array}{ll} \!\!\! \binom{d}{\alpha_1,\dots, \alpha_n}a_1^{s_1-r_1}\! \cdots a_n^{s_n-r_n}b_1^{r_1}\cdots b_n^{r_n} & \text{in the case of differentiation} \\[1ex] a_1^{s_1-r_1}\cdots a_n^{s_n-r_n}b_1^{r_1}\cdots b_n^{r_n} & \text{in the case of contraction.}  \end{array} \right.
\]
\item If there is no directed path from $x^{\alpha}$ to $x_1^{d_1-1}\cdots x_n^{d_n-1}$, then $c_{\alpha}=0$.
\end{itemize}
\end{thm}

\begin{proof}
We give the proof for the differentiation action. The contraction case follows similarly. If we begin by looking at a term of $f_i \circ F$ of the form $(ka_ic_{\alpha} - \ell b_ic_{\beta})X^{\gamma}$, then $\frac{\alpha_1! \cdots \alpha_n!}{k}=\frac{\beta_1!\cdots \beta_n!}{\ell}=\gamma_1!\cdots \gamma_n!$. Introducing $c_{\alpha}'=c_{\alpha}/\binom{d}{\alpha_1,\dots,\alpha_n}$ we then get that
\begin{align*}
(ka_ic_{\alpha} - \ell b_ic_{\beta})X^{\gamma} &= (ka_i\binom{d}{\alpha_1,\dots,\alpha_n}c_{\alpha}' - \ell b_i\binom{d}{\beta_1,\dots,\beta_n}c_{\beta}')X^{\gamma} \\
&= \binom{d}{\gamma_1,\dots, \gamma_n}(a_ic_{\alpha}' - b_ic_{\beta}')X^{\gamma}.
\end{align*}
Hence the relations for all $f_i\circ F$ to vanish in terms of the $c_{\alpha}'$ becomes $a_ic_{\alpha}' = b_ic_{\beta}'$. 

Next, note that the only monomial in $\M_{d_1,\dots,d_n}$ of degree $d$ is $x_1^{d_1-1}\cdots x_n^{d_n-1}$. If there is no path from $x^{\alpha}$ to it in $G_{B,d}$, there must therefore be a path from $x^{\alpha}$ to a directed cycle $C$ starting at some $x^{\beta}$. If $r_i$ is the number of $i$-labeled edges in $C$, then going through the cycle gives that
\[
c_{\beta}' = \frac{b_1^{r_1}\cdots b_n^{r_n}}{a_1^{r_1}\cdots a_n^{r_n}}c_{\beta}'.
\]
Thus
\[
(a_1^{r_1}\cdots a_n^{r_n} - b_1^{r_1} \cdots b_n^{r_n})c_{\beta}'=p(C)c_{\beta}'=0.
\]
But from Lemma \ref{lem:monomial}, we know that $p(C)$ is a factor of $\res(B)$, which is nonzero as $B$ is a complete intersection. Hence $p(C)\neq 0$ and we must have $c_{\beta}'=0$. Finally, as there is a directed path from $x^{\alpha}$ to this $x^{\beta}$, $c_{\alpha}'$ is a multiple of $c_{\beta}'$ and is therefore also zero, showing the first part of the theorem.

For the other coefficients, we first claim that $c_{d_1-1,\dots, d_n-1}'\neq 0$. If not, then $x_1^{d_1-1}\cdots x_n^{d_n-1}$ would annihilate $F$, giving that $x_1^{d_1-1}\cdots x_n^{d_n-1}\in \Ann(F)=(B)$. But as $x_1^{d_1-1}\cdots x_n^{d_n-1}$ is in $\M_{d_1,\dots, d_n}$, it is a basis element of $R/(B)$ by Theorem \ref{thm:basis}, giving a contradiction.

The formula for the $c_{\alpha}$ where $x^{\alpha}$ has a directed path to $x_1^{d_1-1}\cdots x_n^{d_n-1}$ in $G_{B,d}$ is now coherent with the relations for all $f_i\circ F$ to vanish when we only look at edges in the reduction graph $G_{B,d}$. Further, if $a_ic_{\alpha}' = b_ic_{\beta}'$ is a relation coming from an edge of $\widehat{G}_{B,d}$ that are not in $G_{B,d}$, then this will not impose any additional criteria. Indeed, if both $c_{\alpha}'$ and $c_{\beta}'$ were nonzero, $x^{\alpha}$ and $x^{\beta}$ must both have a path to $x_1^{d_1-1}\cdots x_n^{d_n-1}$ in $G_{B,d}$, and going from $x^{\alpha}$ via this path to $x_1^{d_1-1}\cdots x_n^{d_n-1}$ and then following the other path backwards to $x^{\beta}$, gives a relation between $c_{\alpha}'$ and $c_{\beta}'$ imposed only by edges in $G_{B,d}$. This must agree with $a_ic_{\alpha}' = b_ic_{\beta}'$ as else it would force $c_{\alpha}'=0$ and therefore also $c_{d_1-1,\dots, d_n-1}'=0$, which we know is not the case. Similarly, if only one of $c_{\alpha}'$ and $c_{\beta}'$ were nonzero, that would also force $c_{d_1-1,\dots, d_n-1}'=0$. Hence it suffices to look at the edges in the reduction graph $G_{B,d}$ as desired.
\end{proof}

Notice that if we take $b_1= \dots =b_n=0$ in Theorem \ref{thm:dual} we get the well known dual generator 
$F=cX_1^{d_1-1} \cdots X_n^{d_n-1}$ of a monomial complete intersection. 


\begin{ex}
The binomials $f_1=a_1x_1^2 - b_1x_1x_3, f_2=a_2x_2^2 - b_2x_2x_3$ and $f_3=a_3x_3^2 - b_3x_2x_3$ from Example \ref{ex:begining_example} have a reduction graph $G_{B,3}$ as illustrated in Figure \ref{fig:graphexample_deg3}. Looking at the maximal paths that end at $x_1x_2x_3$, we see that an $1$-labeled edge is used at most two times and $2$ and $3$-labeled edges at most one time. Hence the coefficient of $X_1X_2X_3$ in their dual generator $F$ can be taken to be $\binom{3}{1,1,1}a_1^2a_2a_3=6a_1^2a_2a_3$ with respect to differentiation and $a_1^2a_2a_3$ with respect to contraction. As for example $x_1^2x_3$ is one $1$-labeled edge and one $3$-labeled edge away from $x_1x_2x_3$, the monomial $X_1^2X_3$ is given the coefficient $\binom{3}{2,1,0}a_1^{2-1}a_2a_3^{1-1}b_1b_3=3a_1a_2b_1b_3$ with respect to differentiation and $a_1a_2b_1b_3$ with respect to contraction, and as $x_2^3$ has a directed path leading to a cycle, the coefficient of $X_2^3$ in $F$ is zero. Continuing in this way we find 
\begin{align}\label{ex:dual_generator}
\begin{split}
F =& b_1^2a_2b_3X_1^3 + 3a_1b_1a_2b_3X_1^2X_3 + 3a_1b_1a_2a_3X_1^2X_2 \\ &\quad + 3a_1^2a_2b_3X_1X_3^2 + 3a_1^2b_2a_3X_1X_2^2+6a_1^2a_2a_3X_1X_2X_3 
\end{split}
\end{align}
with respect to differentiation and
\begin{align}\label{ex:dual_generator2}
\begin{split}
F =& b_1^2a_2b_3X_1^3 + a_1b_1a_2b_3X_1^2X_3 + a_1b_1a_2a_3X_1^2X_2 \\ &\quad + a_1^2a_2b_3X_1X_3^2 + a_1^2b_2a_3X_1X_2^2+a_1^2a_2a_3X_1X_2X_3 
\end{split}
\end{align}
with respect to contraction.

\begin{figure}[ht]
    \centering
    \[\begin{tikzcd}
	& {x_1^3} & {} & {x_2^3} \\
	& {x_1^2x_3} && {x_2^2x_3} \\
	{x_1^2x_2} & {x_1x_3^2} & {x_1x_2^2} & {x_2x_3^2} \\
	& {x_1x_2x_3} && {x_3^3}
	\arrow["1", from=1-2, to=2-2]
	\arrow["1", from=2-2, to=3-2]
	\arrow["3", from=3-2, to=4-2]
	\arrow["1"', from=3-1, to=4-2]
	\arrow["2", from=3-3, to=4-2]
	\arrow["2"', from=1-4, to=2-4]
	\arrow["3"', from=4-4, to=3-4]
	\arrow["2"', bend right, from=2-4, to=3-4]
	\arrow["3"', bend right, from=3-4, to=2-4]
\end{tikzcd}\]
    \caption{The reduction graph $G_{B,3}$ for $B$ as in Example \ref{ex:begining_example}}
    \label{fig:graphexample_deg3} 
    \vspace{-12pt}
    \qedhere
\end{figure}
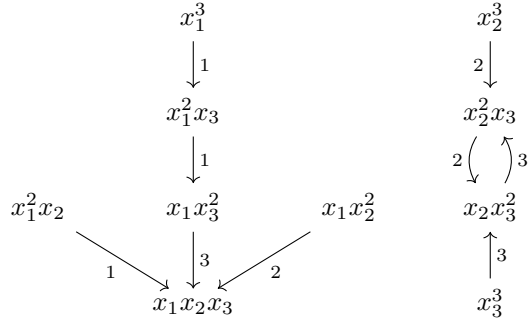
\end{ex}

If we would take the form $F$ in \eqref{ex:dual_generator}, and choose values of $a_1, a_2, a_3, b_1, b_2, b_3$ so that $\res(f_1, f_2, f_3)=0$, then we still get a Gorenstein algebra $R/\Ann(F)$ with $(f_1, f_2, f_3) \subset \Ann(F)$. In the same way, we can take any form $F$ constructed as described in Theorem \ref{thm:dual}, and consider the family of Gorenstein algebras $R/\Ann(F)$ parameterized by $a_1, \ldots, a_n, b_1, \ldots, b_n$. We finish the section by generalizing Theorem \ref{thm:basis} to this larger class of Gorenstein algebras.

\begin{prop}\label{thm:dual_sqfree_basis}
    Let $F$ be as in Theorem \ref{thm:dual}, for some given binomials $f_i=a_ix_i^{d_i}-b_im_i$. Then the set of monomials $\M_{d_1, \ldots, d_n}$ span $R/\Ann(F)$ as a vector space, independent of the choice of coefficients $a_1, \ldots, a_n, b_1, \ldots, b_n$.
\end{prop}
\begin{proof}
Let $B=\{f_1, \ldots, f_n\}$ and take some monomial $m$ of degree $d$. As the reduction graph $G_{B,d}$ is independent of the coefficients chosen, we can argue about its properties using generic choices for $a_1, \ldots, a_n, b_1, \ldots, b_n$ and then apply those properties of the graph to any choice of coefficients. By the definition of $G_{B,d}$, the directed path starting in $m$ must either
\begin{enumerate}
    \item\label{item1:proof_sqfree_gor} lead to a monomial in $\M_{d_1, \ldots, d_n}$, or
    \item\label{item2:proof_sqfree_gor} lead to a directed cycle.
\end{enumerate}
 \parbox{0.95\textwidth}{Claim: \emph{If the directed path starting in $m$ leads into a directed cycle, then the same holds for any monomial divisible by $m$, in its corresponding graph.}}
    
    To prove the claim, note that $m$ satisfying 2 is equivalent to $m \in (B)$ when the coefficients $a_1, \ldots, a_n, b_1, \ldots, b_n$ are chosen generically. Then it is clear that $m'm \in (B)$, for any monomial $m'$ and generic choice of coefficients. Equivalently, the path starting in $m'm$ leads to a directed cycle.  

Let $D=d_1+ \dots +d_n-n$, and take a monomial $m$ of degree $d \le D$ satisfying \ref{item2:proof_sqfree_gor}. For any monomial $m'$ of degree $D-d$ there will be no directed path from $m'm$ to $x_1^{d_1-1} \cdots x_n^{d_n-1}$, by the above proved claim. By Theorem \ref{thm:dual} there are then no terms divisible by $m$ in $F$, so $m \in \Ann(F)$.

Next, note that $(B) \subseteq \Ann(F)$, independently of $\res(B)$. If $m$ is a monomial satisfying \ref{item1:proof_sqfree_gor} then there is a monomial $m' \in \M_{d_1, \ldots, d_n}$ such that $m=\lambda m'$ modulo $(B)$, for some $\lambda \in \C$. Then $m=\lambda m'$ holds in $R/\Ann(F)$. 

We have now proved that any monomial is either equal to zero, or equal to a constant multiple of a monomial in $\M_{d_1, \ldots, d_n}$, in $R/\Ann(F)$.
\end{proof}

\section{The resultant}\label{sec:res}
Having established several properties for complete intersections determined by binomials on normal form, we are now interested in the criteria for when a family of binomials on that form does determine a complete intersection. To this end, we seek the factors of the resultant for a family of binomials on normal form. 

For $i=1, \ldots, n$ let $f_i= \sum_\alpha c_{i,\alpha}x^\alpha$ be homogeneous polynomials  in $R=K[x_1,\dots, x_n]$. Considering the $c_{i,\alpha}$'s as parameters, the resultant $\res(f_1, \ldots, f_n)$ is an irreducible polynomial in $K[c_{i,\alpha}  : \, i=1, \ldots, n, |\alpha|=d_i]$ which is non-zero if and only if the polynomials have a non-trivial common zero in the algebraic closure $\overline{K}$ of $K$ \cite[Theorem 1.6.1]{Dickenstein}. That is, the resultant is non-zero if and only if $(f_1, \ldots, f_n)$ is a complete intersection. 
To calculate the resultant we follow a construction for example described in \cite[Chapter 3]{using_alg_geo}. 
Let
\[
d=\sum_{i=1}^n(d_i-1) + 1
\]
and partition the set of monomials of degree $d$ into $n$ subsets $S_1,\dots, S_n$ where 
\[
S_i= \{ x^{\alpha} \ : \ \alpha_1+\cdots +\alpha_n=d, \: x_i^{d_i}|x^{\alpha}, \text{ and} \ x_j^{d_j}\nmid x^{\alpha} \ \text{for} \ j<i\}.
\]
Note that this is indeed a partition of the monomials of degree $d$ by the choice of $d$. From these we can construct a system of polynomial equations given by 
\begin{align} \label{eq:constr_res}
\frac{x^{\alpha}}{x_1^{d_1}}f_1 &= 0 \quad \text{for all } x^{\alpha} \in S_1 \nonumber \\ 
&\:\:\vdots \\ \nonumber 
\frac{x^{\alpha}}{x_n^{d_n}} f_n &= 0 \quad \text{for all } x^{\alpha} \in S_n. 
\end{align}
All of the polynomials in the system are homogeneous of degree $d$, so ordering the monomials of degree $d$ in some way, we associate a square coefficient matrix $C_n$ to it, whose determinant we denote $|C_n|$. The key property for us is that $\mathrm{res}(f_1,\dots, f_n)$ divides $|C_n|$. 
Indeed, suppose ${\bf p}$ is a zero of $\res(f_1, \ldots, f_n)$. Taking the coefficients $c_{i,\alpha}$ of $f_1, \ldots, f_n$ as prescribed by ${\bf p}$, the $f_i$'s have a nontrivial common zero $({\gamma_1, \ldots, \gamma_n}) \in \overline{K}^n$.  Evaluating \eqref{eq:constr_res} at $({\gamma_1, \ldots, \gamma_n})$  produces a non-trivial solution to the linear system of equations $C_n{\bf v}=0$ and hence $|C_n|=0$. In other words, any zero ${\bf p}$ of $\res(f_1, \ldots, f_n)$ is also a zero of the polynomial $|C_n|$. As $\res(f_1, \ldots, f_n)$ is irreducible, we conclude that it divides $|C_n|$.

Note that specializing to the binomial case means substituting a subset of the $c_{i,\alpha}$'s by 0 in the polynomials $\res(f_1, \ldots, f_n)$ and $|C_n|$. After this operation $\res(f_1, \ldots, f_n)$ is not necessarily irreducible anymore. However, the fact that the resultant divides $|C_n|$ remains true. 

The construction of all $S_i$, and thus also $C_n$, is dependent on our ordering of the variables $x_1,\dots, x_n$. In general, one would need to calculate $|C_n|$ for all $n$ cyclic permutations of the variables and then take the greatest common divisor of all of those determinants to find $\res(f_1,\dots, f_n)$. However, assuming $f_i=a_ix_i^{d_i} - b_im_i$, we will see in Theorem \ref{thm:new_resultant} that one ordering of the variables is enough to determine the radical of the resultant, up to a monomial factor in $a_1, \ldots, a_n$.

Assume again that $f_1,\dots, f_n$ are binomials on normal form,  $f_i=a_ix_i^{d_i} - b_im_i$. In this case $C_n$ is a so called matrix of almost binomial type. See for instance the matrix in Example \ref{ex:almost_binomial_matrix}.

\begin{defn}
Given variables $A=\{a_1,\dots,a_n\}$ and $B=\{b_1,\dots, b_n\}$, a square matrix $C$ is said to be of \emph{almost binomial type} if
\begin{itemize}
\item For each row of $C$, there exists an $i$ such that this row contains exactly one $a_i$ and one $-b_i$ and all other entries are zero, and
\item Every column of $C$ contains exactly one element from $A$. \qedhere
\end{itemize}
\end{defn}

Note that a matrix of almost binomial type differs from a matrix of binomial type as given in \cite[Definition A.1]{HWW} in that the same variable may occur in several different rows, and the sign of $b_i$. That $C_n$ fulfils the first criterion of a matrix of almost binomial type is clear as each $f_i$ only has two terms. For the second criterion, pick a column indexed by some monomial $x^{\alpha}$. From the constructions of the sets $S_1,\dots, S_n$, we know that there is a unique $i$ for which $x^{\alpha} \in S_i$.
Hence the only $j$ for which $a_j$ can appear in that column is $j=i$. Moreover, $a_i$ appears exactly once in the column indexed by $x^{\alpha}$, namely in the row given by the equation $(x^{\alpha}/x_i^{d_i})f_i=0$. Thus $C_n$ is a matrix of almost binomial type.

One benefit of almost binomial type matrices is that we can describe their determinants. To this end, we have the following.

\begin{defn}
A \emph{chain} appearing in $C_n$ is a sequence of entries
\begin{equation*}
a_{j_1} \to -b_{j_1} \to a_{j_2} \to -b_{j_2} \to \cdots \to a_{j_{k-1}} \to -b_{j_{k-1}} \to a_{j_k}
\end{equation*}
where $a_{j_i}$ and $-b_{j_i}$ are in the same row, and $-b_{j_i}$ and $a_{j_{i+1}}$ are in the same column. If $a_{j_k} = a_{j_1}$ are entries at the same position in $C_n$, we call the chain a \emph{circuit}. 
To each circuit we associate a polynomial
\[a_{1}^{r_1}\cdots a_{n}^{r_n} - b_{1}^{r_1}\cdots b_{n}^{r_n},\]
where $r_i$ is the number of times the variable $a_i$ appears in the circuit. 
\end{defn}

The next lemma then follows by similar results for matrices of binomial type.

\begin{lemma}\label{lemma:det_Structure}
For a matrix $C_n$ given by a family of binomials $f_i=a_ix_i^{d_i} - b_im_i$ on normal form, we have that $|C_n|$ factors into a product of a monomial in the $a_1,\dots,a_n$ and polynomials associated to circuits appearing in $C_n$.
\end{lemma}

\begin{proof}
Consider the matrix $C_n'$ obtained from $C_n$ by changing each $a_i$ and $-b_i$ in $C_n$ to $a_{i,j}$ and $b_{i,j}$ in $C_n'$ where $j$ is the row of $C_n$ it lies in. In the new variables $a_{i,j}, b_{i,j}$, the matrix $C_n'$ is of binomial type. 
An expression for $|C_n'|$ is given by \cite[Proposition A.6 and Proposition A.7]{HWW}, and specializing back $a_{i,j}$ and $b_{i,j}$ to $a_i$ and $-b_i$, we get exactly the statement.
\end{proof}

We also have an alternative description of the chains appearing in $C_n$ in terms of the reduction graph $G_{B,d}$.

\begin{lemma} \label{lemma:bijection}
Let $B=\{a_1x_1^{d_1}-b_1m_1, \  \ldots,\  a_nx_n^{d_n}-b_nm_n\}$ be a family of binomials on normal form, $C_n$ the corresponding coefficient matrix, and $G_{B,d}$ the reduction graph for $d=d_1 + \cdots + d_n - n + 1$. Then there is a chain 
\[
a_{j_1} \to -b_{j_1} \to a_{j_2}
\]
in $C_n$ if and only if there is an $j_1$-labeled edge in $G_{B,d}$ from the monomial indexing the column of $a_{j_1}$ the monomial indexing the column of $a_{j_2}$. In particular, a polynomial in $K[a_1, \ldots, a_n, b_1, \ldots, b_n]$ is associated to a circuit appearing in $C_n$ if and only if it is the cycle polynomial of a cycle in $G_{B,d}$.
\end{lemma}
\begin{proof}
Let 
$a_{j_1} \to -b_{j_1} \to a_{j_2}$
be a chain appearing in $C_n$, and say that $a_{j_1}$ lies in a column indexed by $x^{\alpha}$. Then  the vertex $x^{\alpha}$ in $G_{B,d}$  has a unique outgoing edge  to a vertex $(x^{\alpha}/x_{i}^{d_{i}})\cdot m_{i}$ where $i$ is the smallest positive integer such that $x_i^{d_i}$ divides $x^{\alpha}$. From the construction of $C_n$, we know that the smallest such $i$ is $j_1$. Moreover, the row this $a_{j_1}$ lies in is given by the polynomial $(x^{\alpha}/x_{j_1}^{d_{j_1}})\cdot f_{j_1}$, so $(x^{\alpha}/x_{j_1}^{d_{j_1}})\cdot m_{j_1}$ will be the monomial indexing the column where $b_{j_1}$ and $a_{j_2}$ lie. Hence we conclude that the given chain appears precisely if there is an $j_1$-labeled edge in $G_{B,d}$ from the monomial indexing the column of $a_{j_1}$ to the monomial indexing the column of $a_{j_{2}}$. 

Thus there is a bijection between the chains appearing in $C_n$ and the directed paths of $G_{B,d}$ which sends a chain with $r_i$ appearances of $a_i$ to a directed path with $r_i$ edges labeled $i$. In particular, this gives a bijection between polynomials associated to circuits appearing in $C_n$ and cycle polynomials of $G_{B,d}$.
\end{proof}

\begin{lemma}
\label{lemma:chainfactors}
Let $B=\{a_1x_1^{d_1}-b_1m_1, \  \ldots,\  a_nx_n^{d_n}-b_nm_n\}$ be a family of binomials on normal form and $d=d_1 + \cdots + d_n - n + 1$. Then $a_i$ divides $|C_n|$ if and only if there is an $i$-labeled edge not included in a cycle in $G_{B,d}$.
\end{lemma}

\begin{proof}
By Lemma \ref{lemma:bijection}, the graph $G_{B,d}$ is completely determined by the matrix $C_n$ with its labeled columns. Let $F_N$ be the family of directed labeled graphs on $N$ vertices and $N$ edges determined by some almost binomial matrix in the same way. That is, for each $G\in F_N$, there is an almost binomial matrix $M_G$ with columns labeled by the vertices of $G$ with the property that there is an edge $v \overset{i}{\to}w$ in $G$ if and only if there is a chain 
$
a_i \to -b_i \to a_j
$
in $M_G$ where $a_i$ lies in a column indexed by $v$ and $a_j$ in a column indexed by $w$.

We will now prove the stronger statement that if $G \in F_N$ then $a_i$ divides $|M_G|$ if and only if there is an $i$-labeled edge not included in a cycle in $G$. The proof is carried out by induction over $N$, starting at $N=2$.

For $N=2$, there are only two possible matrices of almost binomial type, namely
\[
M_1=\begin{pmatrix}
a_1 & -b_1 \\
-b_1 & a_1
\end{pmatrix} \quad \text{and} \quad
M_2=\begin{pmatrix}
a_1 & -b_1 \\
-b_2 & a_2
\end{pmatrix}.
\]
Both $M_1$ and $M_2$ describe a graph with two vertices having a directed edge from one to the other. In the case of $M_1$, both edges have the same label, while in the case of $M_2$, the labels are different. In any case, there is no edge not included in a cycle. As neither $|M_1|=a_1^2 - b_1^2$ nor $|M_2|=a_1a_2-b_1b_2$ is divisible by $a_1$ or $a_2$, we see that the statement holds for $N=2$.

Assume now that the statement is true for all graphs in $F_{N-1}$ and pick a graph $G\in F_N$. If $G$ is a disjoint union of cycles, then they all correspond to circuits in $|M_G|$ by Lemma \ref{lemma:bijection}, and $|M_G|$ will be a product of the circuit polynomials associated to those circuits. From the definition of a circuit polynomial we know that those are never divisible by any $a_i$.

Next, assume $G$ does have a directed edge outside a cycle. As $G$ only has finitely many vertices, if we follow that edge backwards in some way, we must be able to find a source vertex $v$ of $G$ with no incoming edge and say an $i$-labeled outgoing edge. This means that in the column indexed by $v$ in $M_G$, there is a single nonzero entry $a_i$. Indeed, if there were a $b_j$ in that column, then we could produce a chain $
a_j \to -b_j \to a_i$
in $M_G$, forcing an edge $w \overset{j}{\to}v$ to exist, for some $w$. Now, expanding along the column indexed by $v$, with the only nonzero entry $a_i$,
the matrix we are left with corresponds to the the graph $G \setminus v$ in $F_{N-1}$ where $v$ and its outgoing edge has been removed. That is, $|M_G|=a_i|M_{G\setminus v}|$. By induction, we know that $|M_{G\setminus v}|$ is divisible by an $a_j$ if and only if there is a $j$-labeled edge not in a cycle of $G\setminus v$. Since the edges outside a cycle of $G\setminus v$ together with an extra $i$-labeled edge is the same as the edges outside a cycle of $G$, the statement is true for $G$ and hence for any graph of this form by induction.
\end{proof}

We can now largely describe the factors of the resultant of a family of binomials on normal form.

\begin{thm}\label{thm:new_resultant}
Let $B=\{a_1x_1^{d_1}-b_1m_1, \  \ldots,\  a_nx_n^{d_n}-b_nm_n\} \subset R$ be a family of binomials on normal form and let $d=d_1 + \dots + d_n-n+1$. Then  
\[
\sqrt{\res(B)} = a_1^{t_1} \cdots a_n^{t_n} \sqrt{p(G_{B,d})}
\]  
where $t_i=0$ or $1$, and $t_i=0$ if each $i$-labeled edge is part of a cycle in $G_{B,d}$.
\end{thm}

\begin{proof}
We will show the desired equality using the aid of $|C_n|$. First, as mentioned earlier $\res(B)$ divides $|C_n|$, so
any irreducible factor of $\res(B)$ must occur in $|C_n|$.
Lemma \ref{lemma:det_Structure} describes the possible factors of $|C_n|$. Together with Lemma \ref{lemma:bijection} and Lemma \ref{lemma:chainfactors} we get that
\[
\sqrt{|C_n|} \ \Big{|} \  a_1^{s_1} \cdots a_n^{s_n} \sqrt{p(G_{B,d})}
\]
where $s_i=0$ if each $i$-labeled edge is part of a cycle in $G_{B,d}$ and $s_i=1$ otherwise. In particular $\sqrt{\res(B)}$ divides $a_1^{s_1} \cdots a_n^{s_n} \sqrt{p(G_{B,d})}$. By Lemma \ref{lem:monomial} we know that $\sqrt{p(G_{B,d})}$ divides $\res(B)$, so it follows that 
\[
\sqrt{\res(B)} = a_1^{t_1} \cdots a_n^{t_n} \sqrt{p(G_{B,d})}
\]
for some non-negative integers $t_i \le s_i$, which proves the statement.
\end{proof}

For a large class of families of binomials on normal form, we can give a precise description of the exponents $t_1, \ldots, t_n$ in Theorem \ref{thm:new_resultant}.

\begin{thm}\label{thm:resultant}
Let $B=\{a_1x_1^{d_1}-b_1m_1, \  \ldots,\  a_nx_n^{d_n}-b_nm_n\} \subset R$ be a family of binomials on normal form and let $d=d_1 + \dots + d_n-n+1$. If no $m_k$ is a pure power of $x_i$ then $a_i$ divides $\res(B)$. In particular, if no $m_k$ is a pure power of a variable, then  
\[
\sqrt{\res(B)} = \sqrt{|C_n|} = a_1 \cdots a_n \sqrt{p(G_{B,d})}.
\]  
\end{thm}
\begin{proof}
Assume that no $m_k$ is a pure power of $x_i$, for a fixed $i$. If $a_i=0$, then the point $e_i\in K^n$ with a $1$ in position $i$ and zeroes elsewhere will be a non-trivial common zero of $f_1,\dots, f_n$. Hence they will not be a complete intersection in that case and thus $a_i$ must divide $\res(B)$. Moreover, if no $m_k$ is a pure power of a variable, then all $a_i$'s are factors of $\res(B)$, and hence $\sqrt{\res(B)} =a_1 \cdots a_n \sqrt{p(G_{B,d})}$. By the same argument as in the proof of Theorem \ref{thm:new_resultant} $\sqrt{\res(B)}$ divides $\sqrt{|C_n|}$ who divides $a_1^{s_1} \cdots a_n^{s_n} \sqrt{p(G_{B,d})}$ for some $s_i$'s in $\{0,1\}$. The only possibility is $s_1= \cdots = s_n=1$, forcing the desired equality.
\end{proof}

At a first glance, it might appear that Theorem \ref{thm:resultant} can contradict the $t_i=0$ criterion from Theorem \ref{thm:new_resultant}. But if no $m_k$ is a pure power of $x_i$, then the vertex $x_i^d$ in $G_{B,d}$ has no incoming edge, as every target of an edge in $G_{B,d}$ is divisible by some $m_k$. In particular, $x_i^d$ and its outgoing $i$-labeled edge are not part of a cycle. 

Note also that in Theorem \ref{thm:new_resultant} and Theorem \ref{thm:resultant} the resultant is considered as a polynomial in the $2n$ variables $a_1, \ldots, a_n, b_1, \ldots, b_n$, making the $a_i$'s and the irreducible factors of the cycle polynomials distinct. If we drop this assumption and substitute some of the $b_i$ for zero, then $\sqrt{p(G_{B,d})}$ may be divisible by some of the $a_i$.  
In particular, it might be the case that $\sqrt{\res(B)} \neq a_1\cdots a_n\sqrt{p(G_{B,d})}$. However, following the proof of Theorem \ref{thm:resultant}, it is straightforward to deduce that the equality $\sqrt{\res(B)} =\sqrt{a_1\cdots a_n p(G_{B,d})}$ holds with respect to this substitution.

\begin{ex}\label{ex:almost_binomial_matrix}
Consider the polynomials $f_1=a_1x_1^2 - b_1x_1x_3, f_2=a_2x_2^2 - b_2x_2x_3$ and $f_3=a_3x_3^2 - b_3x_2x_3$ from Example \ref{ex:begining_example}. In this case we calculated that $\sqrt{p(G_{B,4})}=a_2a_3-b_2b_3$. This can be compared with the resultant which we can compute, using Macaulay2 and the package ``Resultants'' \cite{Resultants}, to be $\res(B)=a_1^4a_2^2a_3^2(a_2a_3-b_2b_3)^2$. The matrix $C_3$ is given by

\[
\left(\begin{smallmatrix}
x^4 & x^3y & x^3z & x^2y^2 & x^2yz & x^2z^2 & xy^3 & xy^2z& xyz^2 & xz^3& y^4& y^3z& y^2z^2& yz^3& z^4 \\ \hline
a_1 & 0 & -b_1 & 0 & 0 & 0 & 0 & 0 & 0 & 0 & 0 & 0 & 0 & 0 & 0 \\
0 & a_1 & 0 & 0 & -b_1 & 0 & 0 & 0 & 0 & 0 & 0 & 0 & 0 & 0 & 0 \\
0 & 0 & a_1 & 0 & 0 & -b_1 & 0 & 0 & 0 & 0 & 0 & 0 & 0 & 0 & 0 \\
0 & 0 & 0 & a_1 & 0 & 0 & 0 & -b_1 & 0 & 0 & 0 & 0 & 0 & 0 & 0 \\
0 & 0 & 0 & 0 & a_1 & 0 & 0 & 0 & -b_1 & 0 & 0 & 0 & 0 & 0 & 0 \\
0 & 0 & 0 & 0 & 0 & a_1 & 0 & 0 & 0 & -b_1 & 0 & 0 & 0 & 0 & 0 \\
0 & 0 & 0 & 0 & 0 & 0 & a_2 & -b_2 & 0 & 0 & 0 & 0 & 0 & 0 & 0 \\
0 & 0 & 0 & 0 & 0 & 0 & 0 & \mathbf{a_2} & -\mathbf{b_2} & 0 & 0 & 0 & 0 & 0 & 0 \\
0 & 0 & 0 & 0 & 0 & 0 & 0 & 0 & 0 & 0 & a_2 & -b_2 & 0 & 0 & 0 \\
0 & 0 & 0 & 0 & 0 & 0 & 0 & 0 & 0 & 0 & 0 & a_2 & -b_2 & 0 & 0 \\
0 & 0 & 0 & 0 & 0 & 0 & 0 & 0 & 0 & 0 & 0 & 0 & \mathbf{a_2} & -\mathbf{b_2} & 0 \\
0 & 0 & 0 & 0 & 0 & 0 & 0 & -\mathbf{b_3} & \mathbf{a_3} & 0 & 0 & 0 & 0 & 0 & 0 \\
0 & 0 & 0 & 0 & 0 & 0 & 0 & 0 & -b_3 & a_3 & 0 & 0 & 0 & 0 & 0 \\
0 & 0 & 0 & 0 & 0 & 0 & 0 & 0 & 0 & 0 & 0 & 0 & -\mathbf{b_3} & \mathbf{a_3} & 0 \\
0 & 0 & 0 & 0 & 0 & 0 & 0 & 0 & 0 & 0 & 0 & 0 & 0 & -b_3 & a_3 
\end{smallmatrix}\right)
\]  

where we have put the monomials labeling the columns on top and where the entries that occur in a cycle are in boldface. The labels use $x,y,z$ instead of $x_1,x_2,x_3$ for readability. Calculating its determinant we find that $|C_3|=a_1^6a_2^3a_3^2(a_2a_3-b_2b_3)^2$, and as all $m_i$ are squarefree in this case, we see that \[\sqrt{\res(B)} = \sqrt{|C_3|} = a_1a_2a_3\sqrt{p(G_{B,4})}= a_1a_2a_3(a_2a_3-b_2b_3). \qedhere \]
\end{ex}

\section{Discussion}

\subsection{Presentations of complete intersections}
It is an easy exercise to show that after a suitable linear change of variables, any complete intersection has a generating set on normal form. However, a binomial complete intersection does not necessarily have a generating set of binomials on normal form.

\begin{ex} \label{ex:notnormal}
The binomials
\[
f_1=x_1x_2+x_3x_4, \ f_2=x_1^3+x_2^3, \ f_3 = x_3^3, \ f_4=x_4^3
\]
form a regular sequence. Any generating set of the ideal $(f_1, f_2, f_3, f_4)$ must contain $f_1$, so we seek a linear change of variables which puts $f_1$ on normal form. However, a change of variables can not make $f_1$ into a binomial on normal form. To see this, recall that any quadratic form can be written as $x^\top S x$ where $S$ is a symmetric $n \times n$ matrix. The rank of the matrix $S$ is invariant under change of variables. In our case, the symmetric matrix representing $f_1$ has rank four. But the rank of the symmetric matrix representing any quadratic binomial on normal form has rank at most three, as the binomial is supported in at most three variables.  
\end{ex}

The argument in Example \ref{ex:notnormal} relies on the fact that the generators have mixed degrees. This raises the question whether all equigenerated binomial complete intersections can be presented by binomials on normal form. We believe that this is not the case, motivated by the next example where we have not been able to find an appropriate linear transform. 

\begin{ex}\label{ex:notnormal_eq}
The binomial complete intersection
\[
I=(x_1^2-x_2^2, \ x_3^2+x_4^2, \ x_1x_2+x_3x_4, \ x_1x_3+x_2x_4)
\]
is not on normal form. Is there a linear change of variables under which $I$
is generated by binomials on normal form? 
\end{ex}

Clearly $\M_{2,3,3,3}$ and $\M_{2,2,2,2}$ are not bases for the corresponding quotient rings of Example \ref{ex:notnormal} and Example \ref{ex:notnormal_eq}. However, this can be achieved in both cases by a generic change of variables,  
which also puts the generators on normal form.

Another class of complete intersections where the set $\M_{d_1, \ldots, d_n}$ is a vector space basis can be found in \cite{Ab}. 
Let $\ell_1, \ldots, \ell_n$ be linear forms such that $x_1\ell_1, \ldots, x_n\ell_n$ is a complete intersection. It is proved in \cite[Lemma 2.2]{Ab} that $\M_{2, \ldots, 2}$ is a basis for the quotient ring. This lemma is used to prove the Eisenbud-Green-Harris conjecture \cite{EGH} for regular sequences of products of linear forms. The conjecture states that if a homogeneous ideal $I$ contains a regular sequence $f_1, \ldots, f_n$ of degrees $d_1, \ldots, d_n$, then there is an ideal containing $x_1^{d_1}, \ldots, x_n^{d_n}$ with the same Hilbert series as $I$. 

We remark that one can not simply drop the \emph{binomial} condition in Theorem \ref{thm:basis}, and instead only require the generators to be on normal form. Indeed, $\M_{d_1, \ldots, d_n}$ fails to be a basis even when just replacing one binomial with a trinomial.

\begin{ex}
The polynomials 
\[f_1=x_1^2+x_1x_2+x_1x_3, \quad f_2=x_2^2+x_1x_2, \quad f_3=x_3^2+x_1x_2\] define a complete intersection, but the squarefree monomials are not a basis for $\C[x_1,x_2,x_3]/(f_1,f_2,f_3)$ as $x_1x_2x_3=x_2f_1-x_1f_2$. However, the squarefree monomials are a basis after a change of variables, for example $(x_1, x_2, x_3) \mapsto (x_1, x_2, x_1+x_3)$. 
\end{ex}

\begin{que}  \label{que:hww}
Harima, Wachi, and Watanabe remarked \cite[Remark 3.5]{HWW} that it is plausible that every quadratic complete intersection has the squarefree monomials as a vector space basis after a suitable change of variables. In spirit of our results, we find it natural to extend this question and we ask if every complete intersection has 
$\M_{d_1, \ldots, d_n}$ as a basis after some change of variables?
\end{que}

It might be tempting to believe that one can pose a more general version of Question \ref{que:hww}: If $A$ is a complete intersection, does any set of monomials of a given degree $d$ and of cardinality $\dim_{\C} A_d$ constitute a basis for $A_d$ after a suitable change of variables? The following example shows that this is not the case in general.

\begin{ex}
Consider the algebra $A=\C[x_1,\ldots,x_6]/I$, where $I$ is a quadratic complete intersection. The Hilbert series of $A$ equals $1+6t+15t^2+20t^3+\cdots$,  so
if we pick $16$ quadratic monomials, they will be linearly dependent. That is, we have a non-trivial equality $c_1 m_1 + \cdots + c_{16} m_{16} = 0$. When we multiply this equality with $x_1$, we get that the monomials $x_1 m_1, \ldots, x_1 m_{16}$ are linearly dependent. This means that we cannot pick any set of $20$ monomials as a vector space basis for the quotient in degree three, as we cannot pick these $16$ monomials to be included in such a basis.
\end{ex}

We would also like to point out that not every complete intersection can be realized as a \emph{binomial} complete intersection, after a linear change of variables. Even though this fact seems highly intuitive, we find it informative to provide a concrete example. 

\begin{ex}
    Let $\ell_i=a_ix+b_iy$, $i=1,2,3$, be three pairwise linearly independent linear forms in $\C[x,y]$. Choose a $2 \times 3$-matrix $C=(c_{ij})$ over $\C$ with rank 2 such that 
    \[
    f=c_{11}\ell_1^d+c_{12}\ell_2^d+c_{13}\ell_3^d, \quad g=c_{21}\ell_1^d+c_{22}\ell_2^d+c_{23}\ell_3^d 
    \]
    is a regular sequence, for a given positive integer $d$. Then $(f,g)$ is not a binomial ideal, in any choice of variables. To prove this, first note that any minimal generating set of $(f,g)$ satisfies the above description, also after applying a linear change of variables. So, we want to check if there is a choice of coefficients such that all but two terms vanishes in both $f$ and $g$ on the above form. Choose two monomials $m_1$ and $m_2$ and let $J_f$ be the ideal of $\C[a_1,a_2,a_3,b_1,b_2,b_3,c_{11}, \ldots, c_{23}]$ generated by the coefficients of the monomials not equal to $m_1$ or $m_2$ in $f$. Similarly we pick two monomials for $g$ and define an ideal $J_g$. For $f$ and $g$ to be binomials $\lambda_1m_1+\lambda_2m_2$ and $\lambda_3m_3+\lambda_4m_4$, there should be a point in the variety $V(J_f) \cap V(J_g)$ satisfying the conditions $\ell_1, \ell_2, \ell_3$ being pairwise linearly independent and $C$ having rank 2.     
    Note that increasing $d$ gives more equations defining $V(J_f)$ and $V(J_g)$, while the number of parameters stay the same.
    For our purpose it turns out that $d=8$ is large enough. Primary decompositions of the radical ideals $\sqrt{J_f}$ and $\sqrt{J_g}$ can be computed in Macaulay2. Going through all primary components, for all choices of monomials of degree 8 for $m_1, \ldots, m_4$ reveals that any point in $V(J_f) \cap V(J_g)$ gives $C$ rank less than 2, or makes a pair of $\ell_1, \ell_2, \ell_3$ linearly dependent. Hence $f$ and $g$ can not both be binomials. 
\end{ex}

\subsection{Lefschetz properties of some Gorenstein algebras with squarefree monomial basis}

As we have seen, our class of binomial complete intersections generalize some of the properties of monomial complete intersections. But artinian monomial complete intersections are known to enjoy the strong Lefschetz property \cite{Stanley, Watanabe}, and this is a result that we have not been able to lift to binomial complete intersections on normal form.

Recall that a standard graded Artinian algebra $A=A_0 \oplus A_1 \oplus \dots \oplus A_s$ has the \emph{weak Lefschetz property (WLP)} if there is a linear form $\ell$ for which the induced multiplication maps $\cdot \ell: A_i \to A_{i+1}$ all have maximal rank. Moreover, $A$ has the \emph{strong Lefschetz property (SLP)} if the maps $\cdot \ell^k: A_i \to A_{i+k}$, have maximal rank for all $i$ and $k$.

A longstanding open problem in the study of the Lefschetz properties for graded artinian algebras is whether every complete intersection has the WLP, and this problem is mentioned as early as 1991 in a paper by Reid, Roberts and Roitman \cite{RRR}. But up to this point not much is known.
In fact, the motivation for Harima, Wachi, and Watanabe \cite{HWW} to study quadratic binomial complete intersections was from the perspective of the Lefschetz properties.

When $A$ is a Gorenstein algebra with dual generator $F \in \C[X_1, \ldots, X_n]$, the Lefschetz properties are determined by Hessian matrices. Let $\{g_1, \ldots, g_r\}$ be a basis for the graded component $A_k$. The $k$-th Hessian matrix of $F$, denoted $\Hess^k F$, is the $r \times r$ matrix where the entry on position $(i,j)$ is $(g_ig_j)\circ F$. For a linear form $\ell=c_1x_1+ \dots +c_nx_n$ the matrix $\Hess^k F$, with $X_i$ substituted by $c_i$ for each $i=1, \ldots, n$, represents the multiplication map $\cdot \ell^{D-2k}:A_k \to A_{D-k}$ w.\,r.\,t.\ a certain basis. Here $D$ denotes the socle degree. The connection between Hessian matrices and Lefschetz properties was originally established in \cite{Watanabe_Hess}. For further reading on this topic we also recommend \cite{GZ}.

\begin{ex}\label{ex:gor_ex1}
Let
\begin{align*}
F=&b_{1}b_{3}X_{1}^{3}X_{3}^{2}+2b_{3}X_{1}X_{2}X_{3}^{3}+b_{2}b_{4}X_{2}^{3}X_{4}^{2}+b_{1}b_{4}X_{1}^{2}X_{4}^{3}+2b_{4}X_{2}X_{3}X_{4}^{3}\\
&\quad +2b_{2}X_{1}X_{2}^{3}X_{5}+2b_{1}X_{1}^{3}X_{4}X_{5}+12X_{1}X_{2}X_{3}X_{4}X_{5}+b_{3}b_{5}X_{3}^{3}X_{5}^{2}\\
&\quad +b_{2}b_{5}X_{2}^{2}X_{5}^{3}+2b_{5}X_{3}X_{4}X_{5}^{3}, \quad \text{and}
\end{align*}
\begin{align*}
&f_1=x_1^2-b_1x_2x_3, \quad f_2=x_2^2-b_2x_3x_4, \quad f_3=x_3^2-b_3x_4x_5, \\ 
&f_4=x_4^2-b_4x_1x_5, \quad f_5=x_5^2-b_5x_1x_2. 
\end{align*}
We have $\res(f_1, \ldots, f_5)=(1-b_1b_2b_3b_4b_5)^{11}$, and $\Ann(F)=(f_1, \ldots, f_5)$ when $b_1, \ldots, b_5$ are chosen so that $\res(f_1, \ldots, f_5) \ne 0$, as explained in Theorem \ref{thm:dual}. 
By Theorem \ref{thm:dual_sqfree_basis} the squarefree monomials span the algebra $A=R/\Ann(F)$ as a vector space, and we shall see that they are in fact a basis for any choice of $b_1, \ldots, b_5$. First, the squarefree monomials of degree four are a basis for $A_4$ as
 \[ \hat{x}_1\circ F=X_1, \quad \hat{x}_2\circ F=X_2, \quad \hat{x}_3\circ F=X_3, \quad \hat{x}_4\circ F=X_4, \quad \hat{x}_5\circ F=X_5,
\]
\[
\text{where} \ \  \hat{x}_i= \frac{x_1x_2x_3x_4x_5}{12x_i},
\]
are linearly independent. Letting the squarefree monomials of degree three act on $F$ gives constant multiples of
\begin{align*}
 &b_1X_1^2+2X_2X_3, \quad b_2X_2^2+2X_3X_4, \quad b_3X_3^2+2X_4X_5, \quad b_4X_4^2+2X_1X_5,\\
  &b_5X_5^2+2X_1X_2, \quad X_1X_3, \quad X_1X_4, \quad X_2X_4, \quad X_2X_5, \quad X_3X_5, 
\end{align*}
which are linearly independent. It now follows by symmetry that the Hilbert series is $1+5t+10t^2+10t^3+5t^4+t^5$, for any $b_1, \ldots, b_5$.

The first and second Hessians $\Hess^1$ and $\Hess^2$ w.\,r.\,t.\ squarefree monomial bases represent the multiplication maps $\cdot \ell^3:A_1 \to A_4$ and $\cdot \ell:A_2 \to A_3$ for a generic linear form $\ell$, independent of the value of the resultant. For both matrices, we compute the radical of the ideal in $\C[b_1, \ldots, b_5]$ generated by the coefficients of the determinant. In both cases the radical ideal is the whole ring, which means that both determinants are nonzero, independently of $b_1, \ldots, b_5$.  In other words, the Gorenstein algebra defined by $F$ has the SLP, for any choice of $b_1, \ldots, b_5$. 
\end{ex}

It can also happen that the Hilbert series changes when the coefficients are chosen so that the resultant of the given binomials vanishes. 

\begin{ex}
Let $A=\C[x_1, \ldots, x_5]/\Ann(F)$ where $F$ is the dual generator corresponding to 
\begin{align*}
&f_1=x_1^2-b_1x_2x_5, \quad f_2=x_2^2-b_2x_1x_3, \quad f_3=x_3^2-b_3x_2x_4, \\ 
&f_4=x_4^2-b_4x_3x_5, \quad f_5=x_5^2-b_5x_1x_4 
\end{align*}
in the sense of Theorem \ref{thm:dual}. Here $\res(f_1, \ldots, f_5)=(1-b_1b_2b_3b_4b_5)^5$. It is proved in the same way as in Example \ref{ex:gor_ex1} that $\dim A_1= \dim A_4=5$. To determine the dimension of $A_2$ and $A_3$ we let the squarefree monomials of degree three act on $F$. The computation is done in Macaulay2, and results in the polynomials
\begin{align*}
p_1=&b_1X_1^2+2b_3b_4X_3X_4+2X_2X_5, \quad &q_1&=b_1^2b_2b_5X_1^2+2X_3X_4+2b_1b_2b_5X_2X_5, \\
p_2=&b_2X_2^2+2X_1X_3+2b_4b_5X_4X_5, \quad &q_2&=b_1b_2^2b_3X_2^2+2b_1b_2b_3X_1X_3+2X_4X_5, \\
p_3=&b_3X_3^2+2X_2X_4+2b_1b_5X_1X_5, \quad &q_3&=b_2b_3^2b_4X_3^2+2b_2b_3b_4X_2X_4+2X_1X_5, \\
p_4=&b_4X_4^2+2b_1b_2X_1X_2+2X_3X_5, \quad &q_4&=b_3b_4^2b_5X_4^2+2X_1X_2+2b_3b_4b_5X_3X_5, \\
p_5=&b_5X_5^2+2b_2b_3X_2X_3+2X_1X_4, \quad &q_5&=b_1b_4b_5^2X_5^2+2X_2X_3+2b_1b_4b_5X_1X_4.
\end{align*}
We see that $p_1, \ldots, p_5$ are linearly independent, and when $1-b_1b_2b_3b_4b_5=0$ we have
\begin{align*}
&p_1-b_3b_4q_1=0, \quad p_2-b_4b_5q_2=0, \quad p_3-b_1b_5q_3=0, \quad p_4-b_1b_2q_4=0, \quad \text{and} \\
&p_5-b_2b_3q_5=0.
\end{align*} 
Hence the Hilbert series of $A$ is $1+5t+5t^2+5t^3+5t^4+t^5$ when $1-b_1b_2b_3b_4b_5=0$. We can moreover verify that $A$ has the SLP, for all choices of $b_1, \ldots, b_5$.  
\end{ex}

\begin{que}
Do all Gorenstein algebras, including non-complete intersections, defined by forms from Theorem \ref{thm:dual} have the SLP?
\end{que}

We note that a Gorenstein algebra having a squarefree monomial basis does not imply the WLP.

\begin{ex}
Let $F=X_1X_3^3X_4+X_2X_3X_4^3+X_2^2X_5^3$ define a Gorenstein algebra $A$. The Hilbert series of $A$ is $1+5t+10t^2+10t^3+5t^4+t^5$, but $A$ is not a complete intersection. The second Hessian $\Hess^2$ of $F$ does not have full rank, and therefore $A$ fails the WLP. The set of squarefree monomials is not a basis for $A$, as for example $(x_1x_2x_3x_4x_5)\circ F=0$. However, the squarefree monomials becomes a basis after applying a linear change of variables to $A$. This can be checked in Macaulay2 for example by taking five random linear forms with rational coefficients as new variables. This example is also discussed in \cite[Example 3.9]{HWW_mdual}.
\end{ex}

\section*{Acknowledgement}
We thank Aldo Conca and Junzo Watanabe for comments on a draft of this paper and Thomas Kahle for discussions on the terminology of our objects of study. We also thank the anonymous referee for helpful suggestions.
The second author was partly supported by the grant VR2022-04009. The third author was supported by the grant KAW-2019.0512.

\bibliographystyle{plain}
\bibliography{bci_ref}

\end{document}